\documentclass[reqno,12pt]{amsart}
\usepackage[margin=1in]{geometry}

\usepackage{tikz, pgfplots}
\usetikzlibrary{calc}

\usepackage[font=footnotesize]{caption}

\usepackage{amsthm, amsmath, amssymb,xcolor}

\usepackage[utf8]{inputenc}
\usepackage[T1]{fontenc}

\usepackage{newtxtext,newtxmath}

\usepackage[textsize=small,backgroundcolor=orange!20]{todonotes}

\usepackage[hidelinks]{hyperref}

\usepackage[noabbrev,capitalize]{cleveref}
\crefname{equation}{}{}
\usepackage[color,final]{showkeys} %add in 'final' into parameter to remove showkeys

\colorlet{refkey}{orange!20}
\colorlet{labelkey}{blue!60}

\numberwithin{equation}{section}

\newtheorem*{theorem}{Theorem}

\newcommand{\abs}[1]{\left\lvert#1\right\rvert}

\newcommand{\ang}[1]{\left\langle #1 \right\rangle}
\newcommand{\angs}[1]{\langle #1 \rangle}

\newcommand{\RR}{\mathbb{R}}

\newcommand{\ZZ}{\mathbb{Z}}

\title{Exploring a planet, revisited}

\author{Yufei Zhao}
\address{Department of Mathematics, Massachusetts Institute of Technology, Cambridge, MA 02139, USA}
\email{yufeiz@mit.edu}

\begin{document}

\begin{abstract}
How should we place $n$ great circles on a sphere to minimize the furthest distance between a point on the sphere and its nearest great circle?
Fejes T\'oth conjectured that the optimum is attained by placing $n$ circles evenly spaced all passing through the north and south poles.
This conjecture was recently proved by Jiang and Polyanskii.
We present a short simplification of Ortega-Moreno's alternate proof of this conjecture.
\end{abstract}

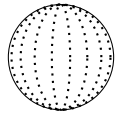
\maketitle

\begin{center}
\begin{tikzpicture}[scale=.7]
	\draw (0,0) circle (1);
	\foreach \x in {10,30,50,70}
	\draw[dotted,thick] (0,0) ellipse ({sin(\x)} and 1);
\end{tikzpicture}
\end{center}

In a classic 1973 \textsc{Monthly} \emph{Research problems} column~\cite{Fej73}, wonderfully titled \emph{Exploring a planet}, L. Fejes T\'oth asked for the most economical way to explore a planet using $n$ satellites. Mathematically, the problem asks to place $n$ great circles on a sphere to minimize the furtherest distance between a point on the sphere and its nearest great circle. He conjectured that the optimal configuration has $n$ evenly spaced great circles all passing through the north and south poles, which he equivalently stated as:
\begin{quote}
If $n$ equal zones cover the sphere then their width is at least $\pi/n$. Here a \emph{zone} of width $w$ is defined as the parallel domain of a great circle of distance $w/2$.	
\end{quote}
This ``zone conjecture'' is a spherical analog of Tarski's plank problem from the 1930's, which asks to show that any covering of a ball by planks (a \emph{plank} is the space between two parallel hyperplanes) must use planks of total width at least the diameter of the ball~\cite{Tar32}.
Tarski gave a beautiful proof of the problem in dimensions 2 and 3 (see the introduction of \cite{KP17} for an exposition of Tarski's proof, which relies on an observation by Archimedes).
The problem in all dimensions was settled some twenty years later by Bang~\cite{Ban50,Ban51} in a stunning proof.
See \cite[Section 3.4]{BMP} for a survey of related problems.

Fejes T\'oth's zone conjecture was recently proved in a beautiful paper of Jiang and Polyanskii~\cite{JP17}. 
Ortega-Moreno~\cite{Ort21}, apparently unaware of Jiang and Polyanskii's work, found another very nice proof of the conjecture.
Amazingly, these two proofs are completely different!
They both prove the result in arbitrary dimensions.
The Jiang--Polyanskii proof builds on the ideas of Bang~\cite{Ban51} and Goodman and Goodman~\cite{GG45}, and it  allows zones of different widths.
Ortega-Moreno's proof, however, is inspired by Ball's solution to the complex plank problem~\cite{Bal01} and uses inverse eigenvectors and trignometric polynomials, though it only works for equal-width zones.

Here we give a streamlined presentation of Ortega-Moreno's proof. 
His proof starts by reformulating the problem in terms of inverse eigenvectors.
We eliminate the need to discuss inverse eigenvectors, thereby giving a shorter and more direct proof. 

The problem in $\RR^d$ is equivalent to showing that given $n$ hyperplanes through the origin in $\RR^d$, there is always a point on the unit sphere with distance at least $\sin\frac{\pi}{2n}$ to every hyperplane. Let $v_1, \dots, v_n$ be the unit normal 
vectors to the hyperplanes. By compactness of the sphere, it suffices to prove the following.

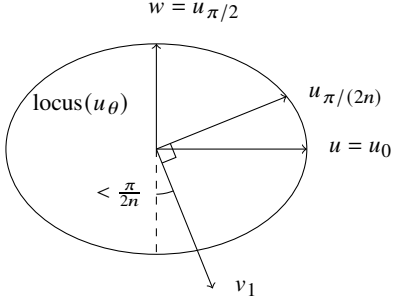
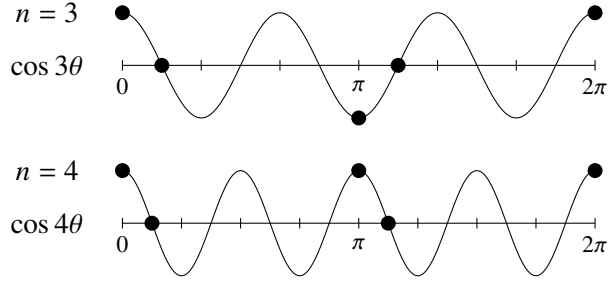
\begin{figure}\centering
\begin{tikzpicture}[scale=2,
					font=\scriptsize,
					baseline={([yshift=-1ex]current bounding box.center)}
					]
	\newcommand\ww{0.7}
	\newcommand\zz{30}
	\newcommand\yy{atan(\ww * sin(\zz)/cos(\zz))}
	\newcommand\ram{.1}
	\newcommand\aar{.3}

	\draw[->] (0,0) --(1,0)  node[label={right:$u = u_0$}] (u) {};
	\draw[->] (0,0) --(0,\ww)  node[label={above:\hspace{7ex}$w = u_{\pi/2}$}] (w){};
	\draw (0,0) ellipse (1 and \ww);
	\draw[->] (0,0) -- ({cos(\zz)}, {\ww * sin(\zz)})node[label=right:$u_{\pi/(2n)}$] {};
	\draw[->] (0,0) -- ({-90+\yy}:1) node[label=right:$v_1$]{};
	\begin{scope}[rotate=\yy]
		\draw (\ram,0)--(\ram,-\ram)--(0,{-\ram});
	\end{scope}
	\draw[dashed] (0,0) -- (0,-\ww);
	\draw (0,-\aar) node[label={[label distance=-.6ex]-180:{\tiny $<\frac{\pi}{2n}$}}]{} arc (-90:-90+\yy:\aar);
	\node at (-.5,.3) {$\operatorname{locus}(u_\theta)$};
\end{tikzpicture}
\qquad 
\begin{tikzpicture}[v/.style={fill,inner sep=2pt,circle},font=\small,
					baseline={([yshift=-1ex]current bounding box.center)}]
\begin{scope}[yscale=.7]
	\newcommand{\nn}{3}
	\draw (0,0)--(2*pi,0);
	\draw plot[domain=0:2*pi, samples=100,smooth] (\x, {cos(deg(\nn*\x))}); 
	\foreach \x/\y in {0/$0$,pi/$\pi$,2*pi/$2\pi$}{
		\draw (\x,.1)--(\x,-.1);	
		\node at (\x,-.4) {\scriptsize\y};
	}
	\foreach \x in {pi/3,2*pi/3,4*pi/3,5*pi/3}{
		\draw (\x,.1)--(\x,-.1);		    
	}
	\foreach \x in {0,{ pi/(2*\nn) },pi,{ pi + pi/(2*\nn) },{2*pi}}
		\node[v] at ({ \x },{ cos(deg(\nn * (\x))) }) {};
	\node at (-1,1) {$n = \nn$};
	\node at (-1,0) {$\cos \nn\theta$};
\end{scope}

\begin{scope}[yscale=.7,yshift=-3cm]
	\newcommand{\nn}{4}
	\draw (0,0)--(2*pi,0);
	\draw plot[domain=0:2*pi, samples=100,smooth] (\x, {cos(deg(\nn*\x))}); 
	\foreach \x/\y in {0/$0$,pi/$\pi$,2*pi/$2\pi$}{
		\draw (\x,.1)--(\x,-.1);	
		\node at (\x,-.4) {\scriptsize\y};
	}
	\foreach \x in {pi/4,2*pi/4,3*pi/4,5*pi/4,6*pi/4,7*pi/4}{
		\draw (\x,.1)--(\x,-.1);		    
	}
	\foreach \x in {0,{ pi/(2*\nn) },pi,{ pi + pi/(2*\nn) }, {2*pi}}
		\node[v] at ({ \x },{ cos(deg(\nn * (\x))) }) {};
	\node at (-1,1) {$n = \nn$};
	\node at (-1,0) {$\cos \nn\theta$};
\end{scope}
\end{tikzpicture}

\medskip

\caption*{\footnotesize (\textsc{Left}) The vectors used in the proof. (\textsc{Right})
The dotted points are known values of $f(\theta)$ overlaid on the plot of $\cos n\theta$. Since $\abs{f(\theta)} < 1$ for all $\theta \notin \ZZ \pi$, the intermediate value theorem shows that $f(\theta)$ and $\cos n\theta$ have at least $2n-4$ additional crossings in $[0,2\pi)$, not counting the ones drawn.
} \label{fig}
\end{figure}

\begin{theorem}
Let $v_1, \dots, v_n$ be unit vectors in $\RR^d$.
If $u$ maximizes
$\prod_{i=1}^n \abs{\ang{v_i, u}}$ among unit vectors,
then $\abs{\ang{v_i, u}} \ge \sin\frac{\pi}{2n}$ for all $i$.
\end{theorem}

\begin{proof}
Suppose for contradiction that $\abs{\ang{v_1, u}} < \sin\frac{\pi}{2n}$ (note that $\abs{\ang{v_1, u}} > 0$ due to the choice of $u$) . Then (see left figure) in the 2-dimensional plane spanned by $\{u, v_1\}$,
we can take a vector $w \perp u$ with $\abs{w} < 1$ such that, setting
\[
u_\theta := (\cos \theta) u + (\sin \theta) w,
\]
one has $u_{\pi/(2n)} \perp v_1$ (picture what happens when $\abs{w}=1$, and then shorten $w$). 
Let
\[
f(\theta) = \prod_{i=1}^n \frac{\angs{v_i, u_\theta}}{\angs{v_i, u}}.
\]
We have $u_{\theta + \pi} = -u_\theta$ and so $f(\theta + \pi) = (-1)^n f(\theta)$. 
Let us focus on the domain $\theta \in [0,\pi)$.
Since $u_0 = u$, we have $f(0) = 1$.
Since $v_1 \perp u_{\pi/(2n)}$, we have $f(\frac{\pi}{2n}) = 0$. 
So $f(\theta) = \cos n\theta$ for $\theta \in\{ 0, \frac{\pi}{2n}\}$.
Since $\abs{w} < 1$, for any $\theta \in (0, \pi)$ 
we have $\abs{u_\theta} < 1$ and thus $\abs{f(\theta)} < 1$ by the maximality hypothesis on $u$.
So $f(\theta) - \cos n\theta$ has sign changes at $\theta = \pi/n, 2\pi/n, \dots, (n-1)\pi/n$ (where $\cos n\theta$ alternates between $\pm 1$), and thus it has at least $n-2$ distinct zeros in $(\pi/n,(n-1)\pi/n)$. 
Combining with the two additional zeros at $\theta = 0, \pi/(2n)$, we see that $f(\theta) - \cos n\theta$ has at least $n$ distinct zeros in $[0, \pi)$, and hence at least $2n$ distinct zeros in $[0, 2\pi)$ (see right figure).

Expanding, for some trignometric polynomial $\psi_1(\theta)$ of degree at most $n-2$,
\[
f(\theta) = \prod_{i=1}^n \frac{\ang{v_i, (\cos\theta) u + (\sin \theta) w}}{\ang{v_i, u}}
= \cos^n \theta 
+ \sum_{i=1}^n \frac{\ang{v_i, w}}{\ang{v_i, u}} \cos^{n-1} \theta\sin \theta + \sin^2 \theta  \psi_1(\theta).
\]
We saw in the previous paragraph that $f(\theta)$ is maximized at $\theta = 0$, and thus
\[
0 = f'(0) = \sum_{i=1}^n \frac{\ang{v_i, w}}{\ang{v_i, u}}.
\]
So the second term in the expansion of $f(\theta)$ above is zero.
Thus
\[
f(\theta) -\cos n\theta 
= \cos^n \theta + \sin^2 \theta \psi_1(\theta) - \cos n\theta
= \sin^2 \theta \psi_2(\theta),
\]
for some trignometric polynomial $\psi_2(\theta)$ of degree at most $n-2$. 
So $f(\theta) -\cos n\theta  = \sin^2 \theta \psi_2(\theta)$ has at most $2n-2$ distinct zeros in $[0,2\pi)$, contradicting the earlier claim.
\end{proof}

\subsection*{Acknowledgments}
The author thanks the referees and the editor for helpful comments.

The author was supported by NSF award DMS-1764176, NSF CAREER award DMS-2044606, a Sloan Research Fellowship, and the MIT Solomon Buchsbaum Fund.


\begin{thebibliography}{10}

\bibitem{Bal01}
Keith~M. Ball, \emph{The complex plank problem}, Bull. London Math. Soc.
  \textbf{33} (2001), 433--442.

\bibitem{Ban50}
Th{\o}ger Bang, \emph{On covering by parallel-strips}, Mat. Tidsskr. B
  \textbf{1950} (1950), 49--53.

\bibitem{Ban51}
Th{\o}ger Bang, \emph{A solution of the ``plank problem.''}, Proc. Amer. Math.
  Soc. \textbf{2} (1951), 990--993.

\bibitem{BMP}
Peter Brass, William Moser, and J\'{a}nos Pach, \emph{Research problems in
  discrete geometry}, Springer, 2005.

\bibitem{Fej73}
L.~Fejes~Toth, \emph{Research problems: {E}xploring a planet}, Amer. Math.
  Monthly \textbf{80} (1973), 1043--1044.

\bibitem{GG45}
A.~W. Goodman and R.~E. Goodman, \emph{A circle covering theorem}, Amer. Math.
  Monthly \textbf{52} (1945), 494--498.

\bibitem{JP17}
Zilin Jiang and Alexandr Polyanskii, \emph{Proof of {L}\'{a}szl\'{o} {F}ejes
  {T}\'{o}th's zone conjecture}, Geom. Funct. Anal. \textbf{27} (2017),
  1367--1377.

\bibitem{KP17}
Andrey Kupavskii and J\'{a}nos Pach, \emph{From {T}arski's plank problem to
  simultaneous approximation}, Amer. Math. Monthly \textbf{124} (2017),
  494--505.

\bibitem{Ort21}
Oscar Ortega-Moreno, \emph{An optimal plank theorem}, Proc. Amer. Math. Soc.
  \textbf{149} (2021), 1225--1237.

\bibitem{Tar32}
Alfred Tarski, \emph{Uwagi o stopniu r{\'o}wnowa{\.z}no{\'s}ci
  wielok\k{a}t{\'o}w [{R}emarks on the degree of equivalence of polygons]},
  Parametr \textbf{2} (1932), 310--314.

\end{thebibliography}
\end{document}